\documentclass[11pt]{article}
\usepackage{hyperref}
\usepackage{amsbsy}
\usepackage[T1]{fontenc}
\usepackage{amsfonts,eufrak}
\usepackage[english]{babel}
\usepackage{a4wide,times}

\usepackage[latin1]{inputenc}
\usepackage{amssymb}
\usepackage{epsfig}
\usepackage{amsbsy}
\usepackage{verbatim}
\usepackage{color}
\usepackage{mathrsfs}
\usepackage{graphicx}
\usepackage{epstopdf}
\usepackage{makeidx}
\usepackage{stmaryrd}
\usepackage{amsmath,amsbsy}
\usepackage{mathabx}
\usepackage{amsthm}
\usepackage{xargs}
\usepackage[capitalise]{cleveref}
\usepackage[prependcaption]{todonotes}
\newcommandx{\fab}[2][1=]{\todo[inline, author={Fab}, linecolor=green,backgroundcolor=red!25,bordercolor=red,#1]{#2}}
\newcommandx{\fabnote}[2][1=]{\todo[author={Fab}, linecolor=red,backgroundcolor=red!25,bordercolor=red,#1]{#2}}

\theoremstyle{plain}
\newtheorem{thm}{Theorem}[section]
\newtheorem{cor}[thm]{Corollary}
\newtheorem{lem}[thm]{Lemma}
\newtheorem{prop}[thm]{Proposition}

\theoremstyle{definition}
\newtheorem{defn}{Definition}[section]
\usepackage{dsfont}

\theoremstyle{remark}
\newtheorem{rem}{\bf Remark}[section]
\theoremstyle{remark}

\newtheorem{com*}{\bf Comment}

\usepackage{fancyhdr}

\makeatletter
\def \newequation#1#2{
   \@definecounter{#1}
   \@namedef{the#1}{\hbox{#2}}
   \@namedef{#1}{$$\refstepcounter{#1}}
   \@namedef{end#1}{
      \eqno \csname the#1\endcsname $$\global\@ignoretrue
      }
}
\makeatother
\newequation{E1}{($E_{b,\sigma,W}$)}
   
\makeatletter
\def \newequation#1#2{
   \@definecounter{#1}
   \@namedef{the#1}{\hbox{#2}}
   \@namedef{#1}{$$\refstepcounter{#1}}
   \@namedef{end#1}{
      \eqno \csname the#1\endcsname $$\global\@ignoretrue
      }
   }
\makeatother
\newequation{hyp3}{($\mathcal{H}_{b,\sigma}$)}

\makeatletter
\def \newequation#1#2{
   \@definecounter{#1}
   \@namedef{the#1}{\hbox{#2}}
   \@namedef{#1}{$$\refstepcounter{#1}}
   \@namedef{end#1}{
      \eqno \csname the#1\endcsname $$\global\@ignoretrue
      }
   }
\makeatother
\newequation{shleg}{(ShLeg)}

\makeatletter
\def \newequation#1#2{
   \@definecounter{#1}
   \@namedef{the#1}{\hbox{#2}}
   \@namedef{#1}{$$\refstepcounter{#1}}
   \@namedef{end#1}{
      \eqno \csname the#1\endcsname $$\global\@ignoretrue
      }
   }
\makeatother
\newequation{kl}{(KL)}

\makeatletter
\def \newequation#1#2{
   \@definecounter{#1}
   \@namedef{the#1}{\hbox{#2}}
   \@namedef{#1}{$$\refstepcounter{#1}}
   \@namedef{end#1}{
      \eqno \csname the#1\endcsname $$\global\@ignoretrue
      }
   }
\makeatother
\newequation{haar}{(Haar)}

\def\R{{\mathbb{R}}}
\def\ER{{\mathbb{R}}}

\def\N{{\mathbb{N}}}
\def\E{{\mathbb{E}}}
\def\ES{{\mathbb{E}}}

\def\B{{\cal B}}
\def\P{{\mathbb{P}}}
\def\PE{{\mathbb{P}}}

\def\ent{K}

\def\b{\beta}

\def\d{\delta}

\def\ve{\varepsilon}

\def\tcg{\textcolor{purple}}

\def\diamd{\mathfrak{d}_d}
\def\fab{\textcolor{magenta}}

\author{Gilles Pag\`es~\thanks{Sorbonne Universit\'e,  Laboratoire de Probabilit\'es, Statistique et Mod\'elisation, UMR~8001, case 158, 4, pl. Jussieu, F-75252 Paris Cedex 5, France. E-mail: \texttt{gilles.pages@sorbonne-universite.fr}}  $\;$and 
Fabien Panloup~\thanks{Université d'Angers, CNRS, LAREMA, SFR MATHSTIC, F-49000 Angers, France. E-mail: \texttt{fabien.panloup@univ-angers.fr}}
}
\title{\bf Kolmogorov-Smirnov  distance and discrepancies {\em versus} Wasserstein distances}

\begin{document}
\maketitle

\begin{abstract}
We establish inequalities that compare the $p$-Wasserstein distance to distances which are built as {suprema} of box measures. More precisely, when the measures are supported on $[0,1]^d$, we obtain sharp upper-bounds of the $p$-Wasserstein distance by {(powers of)} the {(uniform)}  discrepancy. As an application, we retrieve the Pro\"inov Theorem.  When the {two distributions}  are supported {by the whole} $\ER^d$, {their} $p$-Wasserstein distance is upper bounded  {by the product of a (power of) their  Kolmogorov{--}Smirnov ($K$--$S$) distance with the sum of their $p$-moments. Reverse inequalities are established when one of the two distributions has a density, depending on its  ${\cal L}^s$-integrability with respect to the Lebesgue measure for some $s>1$}. 
\end{abstract}

{{\footnotesize \textit{Mathematics Subject Classification:} Primary: 60E15, 60F25  Secondary 65C05, 11K38, 65D32.}


\section{Introduction}
 {For a given norm $|\,.\,|$ on $\ER^d$, the $p$-Wasserstein distance is defined  {when $p\ge 1$} by: for all $\mu$, $\nu\in{\cal P}_p(\ER^d)$, the space of probability distributions on ${\mathcal B}or(\ER^d)$ (Borel sets of $\ER^d$) having (at least) $p$-finite moments, 
$$
{\cal W}_p^{|\,.\,|}(\mu,\nu)=\inf \left\{ \left[\ES \,|X-Y|^p\right]^\frac{1}{p}, \PE_X=\mu,\PE_Y=\nu\right\},
$$
where $\PE_X$ and $\PE_Y$ denote the distributions of $X$ and $Y$ respectively. In the sequel, we will only write ${\cal W}_p$ {(and assume throughout the paper that $p\ge1$)}.
It is well-known that ${\cal W}_p$ metrizes  weak convergence {with convergence of the $p$-moments} on ${\cal P}_p(\ER^d)$ (see $e.g.$ \cite{Villani} for details). }The $p$-Wasserstein distance is now widely used in probabilistic and statistical applications. In statistics, this distance usually produces a robust alternative to Kullback-Leibler divergence taking into account the underlying metric structure. In probability theory, the Wasserstein distance is also widely used for quantifying the rate of convergence to equilibrium or analyzing the robustness of stochastic algorithms.\\

{In this paper we establish inequalities that compare {the $p$-Wasserstein distance to} the Kolmogorov-Smirnov ($K$--$S$) distance and its avatars on the state space $[0,1]^d$, usually called \textit{discrepancies}: for some probability distributions $\mu$ and $\nu$ on ${\mathcal B}or(\ER^d)$, we denote by $D^\star$ the distance defined by
\[
D^{\star}(\mu,\nu) =Ê\sup_{x\in \R^d} \big| \mu\big(]\!]-\infty,x]\!]\big) -  \nu\big((]\!]-\infty,x]\!]\big)\big|,
\]
where $]\!]-\infty,x]\!]=\prod_{i=1}^d (-\infty, x^i]$.
When both distributions of interest are supported by unit hypercubes $[0,1]^d$, the $K$--$S$ distance is called the \textit{star discrepancy} (which explains our notation). We will also consider the uniform discrepancy $D^\infty$ between $\mu$ and $\nu$ defined by
\[
D^{\infty}(\mu,\nu) =Ê\sup_{x,\, y\in \R^d} \big| \mu\big([\![x,y]\!]\big) -  \nu\big([\![x,y]\!]\big)\big|.
\]
where $[\![x,y]\!]=\prod_{i=1}^d [x^i, y^i]$ when $x^i\le y^i$ for every $i\in\{1,\ldots, d\}$ and $[\![x,y]\!]=\emptyset$ otherwise. Discrepancy is an important setting, closely related with Quasi-Monte Carlo method (and optimal quantization theory, see Section 3.3 of \cite{luschgy-pages}) where the empirical measure(s) associated to an $n$-tuple or a sequence of $[0,1]^d$-valued vectors is  used to approximate the uniform distribution ${\cal U}([0,1]^d)$ in order to replace sequences of pseudo-random numbers for the computation of integrals or expectations in Numerical Probability (see~\cite{Nied1992,GilPag2026}).  In these inequalities, special attention is paid to the constant to challenge specific  results from QMC theory like {Pro\"inov}'s Theorem when one of the two distributions is the empirical measure  associated to an $n$-tuple.} As well, $K$--$S$-distance, is certainly connected with the 
non parametric goodness of fit Kolmogorov-Smirnov test   which is devoted to testing equality between two distributions, one being known or not (see \emph{e.g.}~\cite[Sec 16.2]{lehmann-romano}).

The objective is thus to provide precise estimates between ${\cal W}_p$ and $D^\star$ or $D^\infty$. Before stating our results, let us remark that the topologies induced by these distances are slightly different. While ${\cal W}_p$ is a metric for  the weak convergence in ${\cal P}_p(\ER^d)$ with convergence of $p$-moments, $D^\star$ and $D^\infty$ apply without conditions on the moments but their induced topology is finer than the weak convergence topology as illustrated by the following counterexample: if 
\begin{equation}\label{eq:counterexampleb}
\nu_n = \frac 12 \big(\delta_{\frac 12} +\delta_{\frac 12(1+\frac 1n)}\big),\; n\ge 1, \quad\mbox{ and }\quad\nu = \delta_{\frac 12}
\end{equation}
then $\nu_n $ weakly converges toward $\nu$ but $D^\star (\nu_n,\nu){= D^\infty (\nu_n,\nu)}=\frac 12$ for every $n\ge 1$.
Consequently, controlling discrepancies by Wasserstein distances will require an absolute continuity assumption  on one of the two distributions under consideration. Conversely,
we will need  some additional moment assumptions when controlling the Wasserstein distance by $D^\star$ or $D^\infty$.
 
\smallskip
 \noindent \textbf{Contributions and plan of the paper.} We first focus on the control of the Wasserstein distance by discrepancies (or $K$--$S$-distance). For this part, we will rely on the  inspiring papers by~\cite{Dereich13} and~\cite{FG} which }{establish} {universal upper-bounds for the  Wasserstein distances  based on a telescopic splitting of the }{distributions}. Section~\ref{sec:2} is devoted to upper-bounding  the Wasserstein distance ${\cal W}_p$ by the uniform discrepancy for  $[0,1]^d$-supported  distributions with a special attention paid to the values of the semi-universal constants depending on $p$ and the dimension $d$. In view of the optimization of the inequalities, we provide a  variant of  the estimates of~\cite{Dereich13} and~\cite{FG} (see Inequality \eqref{eq:FG3} of~\cref{prop:lemmefondaFG}) based on a slight modification of the coupling scheme proposed in~\cite{Dereich13}. These estimates allow us to state our first estimate in~\cref{prop:WvsDisc} for $[0,1]^d$-supported distributions. At first reading, it can be  summed up as follows: for a given norm $|\,.\,|$ on $\ER^d$ and a given $p\ge 1$, a constant $K_{p,d}$ (which is made explicit in the result) exists such that  
 \[
{\cal W}_p(\mu, \nu) \le K_{p,d} \big(D^{\infty}(\mu,\nu)\big)^{\frac 1p\wedge \frac 1d}.
\]
Owing to an optimization strategy, the constant $K_{p,d}$ is then refined in~\cref{prop:refinement} in the special case $p=1$. Extending to $D^\star$  with the help of the {standard} inequality $D^\infty\le 2^d D^\star$ (see \eqref{eq:discbound} for background), this result allows to retrieve a celebrated inequality from Quasi-Monte-Carlo theory with some slightly larger but more universal constants (see~\cref{sec:proinov} and~\cref{rem:universality} for details) a.k.a.  Pro\"inov's theorem (see~\cite{proinov1988}).

 In~\cref{subsec:3.1} , we extend the bounds to the whole space $\ER^d$ and obtain the following {typical } bound  when $\mu$ and $\nu$ have finite moments or order $q>p$ (see~\cref{prop:vitessesur Rd} for a precise setting)
 \[
{\cal W}_p (\mu,\nu)\le K_{p,d,q} D^\infty(\mu,\nu)^{\frac 1d\wedge (\frac{1}{p}-\frac 1q)}.
\]
 With the inequality $D^\infty\le 2^d D^\star$ (which also holds on $\ER^d$), the above bound also holds with respect to the $K$--$S$-distance $D^\star$.\\
 
 Finally, we consider the reverse problem in~\cref{subsec:KSvsW_1unbound}, $i.e.$: bounding $D^\infty$ or $D^\star$ by the Wasserstein distance. In~\cref{thm:boundDbyW}, we show that if $\mu$ or $\nu$ has a density $g$ \emph{w.r.t.} the Lebesgue measure $\lambda_d$ on $\ER^d$,  the following type of result holds: 
$$D^\star(\mu,\nu)\le K_{r,d}{\cal W}^{\ell^\infty}_1(\mu,\nu)^{\frac{d}{r+d}}, $$
where $r>1$ depends on the moments of $g$ and $r=1$ if $g$ is bounded. This bounded case has been already proved in~\cite{GauLi} but with {non}-explicit constants. This reverse inequality is only written for ${\cal W}_1$-distance since it is based on the dual Kantorovich-Rubinstein representation but certainly extends to ${\cal W}_p$ since ${\cal W}_1\le {\cal W}_p$.

\section{Discrepancies for $[0,1]^d$-supported distributions}\label{sec:2}
{As mentioned in the introduction, we first consider $[0,1]^d$-supported distributions and will investigate the more general case of non-compactly supported probability measures in Section~\ref{sec:3}.}

\subsection{Definitions, notation and a technical lemma}
{We define the partial order on $\ER^d$ as follows:} for $d\! \in \N$, $x= (x^1,\ldots,x^d)$, $y= (y^1,\ldots,y^d)\!\in \ER^d$,
\[
x \preceq y \quad \mbox{ if }\quad x^i \le y^i \mbox{ for every $i\!\in \{1, \ldots,d\}$.}
\]
We can define the closed and semi-open {\em boxes} as follows: when  $x \preceq y$,
\begin{align*}
[\![x,y]\!] = \{u\in [0,1]^d: x^i \le  u^i \le y^i, \, i=1,\ldots,d\}= \prod_{i=1}^d [x^i, y^i]\\
]\!]x,y]\!] = \{u\in [0,1]^d: x^i <  u^i \le y^i, \, i=1,\ldots,d\}= \prod_{i=1}^d (x^i, y^i].
\end{align*}
and otherwise ({\em i.e.} if {$x\npreceq y$}), $[\![x,y]\!] = \varnothing$. 
Note that $[\![x,y]\!]$ is also empty whenever $x^i= y^i$ for some index $i$.\\

\noindent {When  $\mu$ and $\nu$ are two probability measures on $([0,1]^d, {\cal B}or([0,1]^d),\lambda_d)$, the {\em uniform} and {\em star} discrepancy  between $\mu$ and $\nu$ (introduced in the first section) take the form:}
\[
D^{\infty}(\mu,\nu) =Ê\sup_{x,\, y\in [0,1]^d} \big| \mu\big([\![x,y]\!]\big) -  \nu\big([\![x,y]\!]\big)\big|.
\]
and
\[
D^{\star}(\mu,\nu) =Ê\sup_{x\in [0,1]^d} \big| \mu\big([\![0,x]\!]\big) -  \nu\big([\![0,x]\!]\big)\big|.
\]
It is classical background (see {\em e.g}~\cite{Nied1992}) that both $D^\infty$ and $D^\star$ are  $[0,1]$-valued strongly equivalent distances on the set of probability measures on $[0,1]^d$ since
\begin{equation}\label{eq:discbound}
D^{\star}\le D^{\infty}\le 2^d D^{\star}\tcg{,} 
\end{equation}
(see {\em e.g.}~\cite{Nied1992})
but whose {{ induced topology is not that   of weak convergence of distributions on $[0,1]^d$}, as emphasized in \eqref{eq:counterexampleb}}. However { if the generalized c.d.f of $\mu$ defined by $F_{\nu}(x)= \nu([\![0,x]\!])$ is continuous} then 
\begin{equation}\label{eq:contre-exemple} 
\nu_n\stackrel{w}{\longrightarrow}\nu \quad \mbox{ if and only if } \quad {D^\star}(\nu_n,\nu)\to 0 \quad \mbox{ as }\quad n\to +\infty.
\end{equation}
The continuity of $F_\nu$ is equivalent  to the fact that, if $(e^i)_{i=1:d}$ denotes the canonical basis of $\R^d$
\[
\forall\, x\!\in [0,1]^d, \; \forall\, i\!\in \{1,\ldots,d\},\quad \nu\big(x+(e^i)^\perp\big)=0
\]
where $(e^i)^\perp:=\{y\!\in \R^d : y^i=0\}$. Note that for convenience we may extend any measure on $([0,1]^d, {\cal B}or([0,1]^d),\lambda_d)$ into a measure on $(\R^d, {\cal B}or(\R^d),\lambda_d)$ by setting $\nu(A)= \nu(A\cap [0,1]^d)$. 

The particular case where  $\nu$ has a continuous c.d.f., especially when $\nu= {\cal U}([0,1]^d)$,  and $\mu$ is the empirical measure {of } a random or deterministic  $n$-tuple whose components are $[0,1]^d$-valued, has been extensively investigated since the 1950s motivated by the so-called Quasi-Monte Carlo method (QMC, see~\cite{KuiNied, Nied1992}).

\medskip This suggests and justifies to compare in a general framework discrepancies  and Wasserstein distances ${\cal W}_p$, $1\le p<+\infty$  in a strong sense. To be more precise we will upper-bound these Wasserstein distances without  any a priori restrictions on the distributions beyond the existence of finite $p$-moments  whereas, for the reverse bounds, we will assume that (at least) one of the two distributions is absolutely continuous {(\textit{w.r.t.} the Lebesgue measure)} to avoid the above counterexample~\eqref{eq:counterexampleb}.

First we need {the following technical lemma whose proof is postponed to \cref{proof:lemma:2point1}}. 
\begin{lem}\label{lem:tech1}$(a)$ Let $\mu$ and $\nu$ two probability measures on $([0,1]^d, {\cal B}or([0,1]^d),\lambda_d)$. Then
\[
D^{\infty}(\mu,\nu) \ge  \sup_{x,\, y\in [0,1]^d} \big| \mu\big(]\!]x,y]\!]\big) -  \nu\big(]\!]x,y]\!]\big)\big|.
\]
$(b)$ If furthermore, $\mu\big((0,1]^d\big)= \nu\big((0,1]^d\big) =1$, then 
\[
D^{\infty}(\mu,\nu) = \sup_{x,\, y\in [0,1]^d} \big| \mu\big(]\!]x,y]\!]\big) -  \nu\big(]\!]x,y]\!]\big)\big|.
\]
$(c)$ Without the additional assumption of $(b)$, we have:
\[
D^{\infty}(\mu,\nu) = \sup_{x,\, y\in [0,1]^d, x\preceq y} \big| \mu\big(]\![x,y]\!]\big) -  \nu\big(]\![x,y]\!]\big)\big|,
\]
where $]\![x,y]\!]$ is defined by 
$$]\![x,y]\!]=\begin{cases}
]\!]x,y]\!]\textnormal{ if $x^i>0,\,\forall \, i\in\{1,\ldots,d\}$}\\
 \left\{u\in[0,1]^d, x^i<u^i\le y^i \textnormal{ if  $x^i>0$}, 0\le u^i\le  y^i \textnormal{ if  $x^i=0$}\right\}\textnormal{ otherwise.}
 \end{cases}
 $$

\end{lem}

\subsection{Bounding Wasserstein distances by the uniform discrepancy} To achieve our first goal, we will rely on the following bounds for the Wasserstein distances: \eqref{eq:FG2} is mainly adapted from Lemma~5 in~\cite{FG} whereas \eqref{eq:FG3} also uses ideas from former results contained in  \cite{Dereich13}.

\begin{prop}[Existing upper-bound and a variant]\label{prop:lemmefondaFG}
$(a)$ Let $\mu$ and $\nu$ two probability measures on $\big([0,1]^d, {\cal B}or([0,1]^d),\lambda_d\big)$ be such that $\mu\big((0,1]^d\big)= \nu\big((0,1]^d\big) =1$. Then, 
\begin{equation}\label{eq:FG2}
{\cal W}^p_p(\mu, \nu)\le\diamd^p\frac{2^p+1}{2}\sum_{\ell= 1}^{+\infty} 2^{-p\ell} \sum_{F\in {\cal P}_\ell}|\mu(F)-\nu(F)|
\end{equation}
where  $\displaystyle \diamd=\sup_{x,y\in (0,1]^d}|y-x|$ depends on $p$, $d$ and the norm $|\cdot|$ on $(0,1]^d$ and
\begin{equation}\label{eq:defpel}
{\cal P}_\ell= \Big\{a+(-2^{-(\ell+1)}, 2^{-(\ell+1)}]^d,\; a = \frac{2\mathbf{k}+\mbox{\bf 1}}{2^{\ell+1}},\,\mathbf{k}\!\in\{0,\ldots,2^{\ell}-1\}^d  \Big\}
\end{equation}
with $\mbox{\bf 1} = (1,Ê\ldots,1)$. Note that ${\rm card}({\cal P}_\ell)= 2^{d\ell}$. We also have for any $\ell_0\in\mathbb{N}^*$,
\begin{equation}\label{eq:FG3}
{\cal W}^p_p(\mu, \nu)\le{\diamd^p}\bigg(\frac{2^p+1}{2}\sum_{\ell= 1}^{\ell_0} 2^{-p\ell} \sum_{F\in {\cal P}_\ell}|\mu(F)-\nu(F)|+2^{-p\ell_0}\bigg).
\end{equation}

\noindent $(b)$ When $\mu$ and $\nu$ are probability measures on $([0,1]^d, {\cal B}or([0,1]^d),\lambda_d)$, then~\eqref{eq:FG2} still holds with the family of $(\tilde{{\cal P}_\ell})_{\ell\ge0}$, where $\tilde{{\cal P}_\ell}$ is a partition of $[0,1]^d$ which only differs from ${{\cal P}_\ell}$ for  the semi-open boxes with a multi-index $\mathbf{k}=(k_1,\ldots,k_d)\!\in\{ 0,\ldots,2^{\ell}-1\}^d$ for which there exists $i\!\in\{1,d\}$ such that $k_i=0$. When such is the case, the semi-open box $]\!] \frac {\mathbf{k}}{2^{\ell}}, \frac {\mathbf{k}+ \mathbf{1}}{2^{\ell}}]\!] = \frac{2\mathbf{k}+\ \mathbf{1}}{2^{\ell+1}}+(-2^{-(\ell+1)}, 2^{-(\ell+1)}]^d$ is replaced\footnote{More simply, when a semi-open box of ${\cal P}_\ell$ has one or several faces which are included in the faces of $[0,1]^d$, we add them to define the elements of $\tilde{{\cal P}_\ell}$.} by
$$
\Big]\!\Big[ \frac {\mathbf{k}}{2^{\ell}}, \frac {\mathbf{k}+ \mathbf{1}}{2^{\ell}}\Big]\!\Big]:= \left\{u\in[0,1]^d, \frac{k_i}{2^{\ell}}<u^i\le \frac{k_i+1}{2^{\ell}} \textnormal{ if  $k_i\ge 1$}, 0\le u^i\le \frac{1}{2^{\ell}} \textnormal{ if  $k_i=0$}\right\}.
 $$
 \end{prop}
\begin{proof} $(a)$  {\sc Step~0}. Inequality \eqref{eq:FG2} is a straightforward adaptation of~\cite[Lemma 5]{FG} written for the canonical Euclidean norm in the set $(-1,1]^d$. For Inequality \eqref{eq:FG3}, one needs to slightly modify~\cite[Lemma 2]{Dereich13} by introducing a sequence $(\hat{\cal P}_\ell)_{\ell\ge0}$ of partitions built as follows:
\begin{itemize}
\item{} For $\ell \in\llbracket 0,\ell_0\rrbracket$, $\hat{\cal P}_\ell={\cal P}_\ell$,
\item{} For $\ell \ge \ell_0+1$ and a given integer $K\ge 2$, $\hat{\cal P}_\ell$ is deduced from $\hat{\cal P}_{\ell-1}$ by dividing each element of $\hat{\cal P}_{\ell-1}$ into $K^d$ new elements. More precisely,
\begin{equation}\label{eq:defphatel}
\hspace{-1cm} \hat{{\cal P}_\ell}= \Big\{a+(-2^{-\ell_0-1} K^{-(\ell-\ell_0)}, 2^{-\ell_0-1} K^{-(\ell-\ell_0)}]^d,\; a = \frac{2{\bf k}+\mbox{\bf 1}}{2^{\ell_0+1} K^{\ell-\ell_0}},\,{\bf k}\!\in
\{0,\ldots,2^{{\ell_0}}K^{\ell-\ell_0}-1\}^d  \Big\}.
\end{equation}
\end{itemize}
We have ${\rm Card}(\hat{\cal P}_\ell)=2^{d\ell_0}K^{d(\ell-\ell_0)}$. 
Since for any $\ell\ge1$, $\hat{\cal P}_\ell$ is built by partitioning each set of $\hat{\cal P}_{\ell-1}$, the proof of~\cite[Lemma 2]{Dereich13} still works. More precisely, noting that the diameter of an element of ${\hat{\cal P}_\ell}$ is $2^{-\ell}\diamd$ when $\ell\le \ell_0$ and $2^{-\ell_0} K^{-(\ell-\ell_0)}$ when $\ell\ge\ell_0+1$.

\noindent {{\sc Step~1}.  First assume that $\mu$ and $\nu$ satisfy the condition
\[
\forall\,C \!\in {\hat{\mathcal{P}} = \bigcup_{\ell \ge1}\hat{{\cal P}_\ell}}, \quad \nu(C)>0 \Longrightarrow \mu(C)>0.
\] 
with the convention $\frac 00=0$}. Then, a careful reading of the proof of~\cite[Lemma 2]{Dereich13}
 leads to (where $L$ denotes a stopping time defined in the proof of this lemma):
\begin{align*}
{\cal W}_p^p(\mu,\nu)&\le \frac{\diamd^p}{2} \mathbb{E}[2^{-pL}{\bf 1}_{\{L\le \ell_0\}}+ 2^{-p\ell_0}K^{-p(L-\ell_0)}{\bf 1}_{\{L> \ell_0\}}]\\
&\le\frac{\diamd^p}{2}\sum_{\ell=0}^{\ell_0} 2^{-p\ell} \sum_{F\in\hat{\cal P}_\ell}\nu(F) \sum_{C \textnormal{ child of } F}\left|\frac{\nu(C)}{\nu(F)}-\frac{\mu(C)}{\mu(F)}\right|\\
&+ \frac{\diamd^p}{2}\sum_{\ell=\ell_0+1}^{+\infty} 2^{-p\ell_0}K^{-p(\ell-\ell_0)} \sum_{F\in\hat{\cal P}_\ell}\nu(F) \sum_{C \textnormal{ child of } F}\left|\frac{\nu(C)}{\nu(F)}-\frac{\mu(C)}{\mu(F)}\right|.
\end{align*}
At this stage, we use the argument from~\cite[Lemma 5]{FG}: noting that 
$$
\nu(F)\left|\frac{\nu(C)}{\nu(F)}-\frac{\mu(C)}{\mu(F)}\right|\le|\nu(C)-\mu(C)|+\frac{\mu(C)}{\mu(F)}|\mu(F)-\nu(F)|,
$$
and setting 
$$ 
\delta_\ell=\sum_{F\in\hat{\cal P}_\ell}|\mu(F)-\nu(F)|,
$$
we get
\begin{align*}
{\cal W}_p^p(\mu,\nu)&\le\frac{\diamd^p}{2}\left(\sum_{\ell=1}^{\ell_0} \delta_\ell (2^{-p\ell}+2^{-p(\ell-1)})+ \delta_{\ell_0+1} (2^{-p\ell_0}+ 2^{-p\ell_0} K^{-p})\right)\\
&+\frac{\diamd^p}{2}\left(\sum_{\ell=\ell_0+2}^{+\infty} \delta_\ell (2^{-p\ell_0} K^{-p(\ell-\ell_0-1)}+2^{-p\ell_0} K^{-p(\ell-\ell_0)})\right)\\
&\le \frac{\diamd^p}{2}(2^p+1)\sum_{\ell=1}^{\ell_0} 2^{-p\ell} \delta_\ell+\diamd^p2^{-p\ell_0}+O(K^{-p}),
\end{align*}
 where, in the last line, we used that $\delta_\ell\le 2$ for any $\ell\ge1$. The result follows by letting $K$ go to $+\infty$.
 
 \noindent {\sc Step~2}. To get rid of the  above weak absolute continuity assumption on $\mu$ and $\nu$, we introduce for $\varepsilon\!\in (0,1)$, $\mu_\ve= \ve \nu +(1-\ve)\mu$.   Then
 \begin{equation*}
{\cal W}^p_p(\mu_\ve, \nu)\le{\diamd^p}\bigg(\frac{2^p+1}{2}\sum_{\ell= 1}^{\ell_0} 2^{-p\ell} \sum_{F\in {\cal P}_\ell}|\mu_\ve(F)-\nu(F)|+2^{-p\ell_0}\bigg).
\end{equation*}
It is clear that $\mu_\ve$ converges in total variation to  $\mu$ so that the finite sum in the right hand side of the above inequality converges to that of~\eqref{eq:FG3}. On the other hand,  as $Z_\ve Y + (1-Z_\ve)X$ where $X\sim\mu$, $Y\sim\nu$ and $Z\sim {\cal B}(\{0,1\},\ve)$, independent of $X$, $Y$ has distribution $\ve \nu +(-\ve)\mu$, one checks that
\[
{\cal W}^p_p(\mu_\ve, \mu)\le \E\, |Z_\ve Y + (1-Z_\ve)X-X|^p= \EÊ|Z_\ve |^p \E\, |X-Y|^p=Ê\ve \E\, |X-Y|^p\xrightarrow{\ve\rightarrow0} 0.
\]
 Hence $|{\cal W}_p(\mu_\ve, \nu)-{\cal W}^p(\mu, \nu)|\le {\cal W}^p(\mu_\ve, \mu)\to 0$ as $\ve \to 0$ which establishes~\eqref{eq:FG3}.

 \smallskip
\noindent $(b)$ One first checks that the coupling argument of~\cite[Lemma 2]{Dereich13} is still true with $[0,1]^d$ and the family of partitions $(\tilde{P}_\ell)_\ell$. Hence,~\cite[Lemma 5]{FG} whose proof is based on this lemma and on arguments which do not depend on the space and the  partition, also extends to $[0,1]^d$.
\end{proof}
From~\cref{prop:lemmefondaFG}, we can deduce the following upper-bounds of the $p$-Wasserstein distance by the uniform discrepancy.
\begin{thm}[Bounding Wasserstein distance by the uniform discrepancy]  \label{prop:WvsDisc} Let $\mu$ and $\nu$ two probability measures on $([0,1]^d, {\cal B}or([0,1]^d),\lambda_d)$.  Then,
\begin{itemize}
\item If $p>d$ then
\begin{align*}
{\cal W}_p^p(\mu, \nu)& \le   \frac{\diamd^p(2^p+1)}{2(2^{p-d}-1)}    D^{\infty}(\mu,\nu) .
\end{align*}
If furthermore, $1+2^{-p}-2^{p-d}>0$, i.e. $p< d+\frac{\log(1+\sqrt{1+2^{-(d-2)}})}{\log 2}-1$, then
\begin{align*}
{\cal W}_p^p(\mu, \nu)& \le   \frac{\diamd^p}{2^{p-d}-1}\left[\frac{2^p+1}{2}    D^{\infty}(\mu,\nu) -  2^{p(1-\frac{1}{d})}\left(1+2^{-p}-2^{p-d}\right) D^{\infty}(\mu,\nu)^{\frac{p}{d}}\right].
\end{align*}

\item If $p=d$  then
\[
{\cal W}_p^p(\mu, \nu) \le  \diamd^d   \bigg( \bigg(   \frac{(d+1)2^{d-1}}{d}+\frac 1{2d}\bigg) D^{\infty}(\mu,\nu)+\frac{2^d+1}{2d\log 2} D^{\infty}(\mu,\nu)\log\Big(\frac{1}{ D^{\infty}(\mu,\nu)}\Big)\bigg).
\]
\item If $p<d$ then
\[
{\cal W}_p^p(\mu, \nu) \le \diamd^p    {2^{-\frac pd}}   \bigg( \frac{2^p+1}{1-{2^{p-d}}} + 2^{p}\bigg) D^{\infty}(\mu,\nu)^{\frac pd}.
\]
\end{itemize}

\end{thm} 

\begin{proof} We first prove the result when $\mu$ and $\nu$ are supported by $(0,1]^d$. 

\noindent {\sc Step~1}: {$\mu\big((0,1]^d\big)= \nu\big((0,1]^d\big) =1$}. In this case,   the elements of ${\cal P}_{\ell}$ (defined by \eqref{eq:defpel}) are all semi-open boxes.
 Hence,  we  deduce using {Lemma~\ref{lem:tech1}}  that, for every $\ell\ge1$, 
\begin{align}
 \sum_{F\in {\cal P}_\ell}|\mu(F)-\nu(F)|&\le \min\Big( \sum_{F\in {\cal P}_\ell}\mu(F)+\nu(F), 2^{d\ell}D^{\infty}(\mu,\nu)\Big)\nonumber\\
 & = \min \Big(2, 2^{d\ell}D^{\infty}(\mu,\nu)\Big).\label{ineq:basedinfty}
\end{align}
Hence {by \eqref{eq:FG2} and \eqref{eq:FG3}, for any  $\ell_0\in\mathbb{N}\cup\{+\infty\}$},
\begin{align}\label{eq:refwp}
{\cal W}_p^p(\mu, \nu) & \le\diamd^p \left(\frac{2^p+1}{2}\sum_{\ell=1}^{\ell_0} 2^{-p\ell}  \min\big(2,  2^{d\ell} D^{\infty}(\mu,\nu)\big)+2^{-p\ell_0}\right).
\end{align}
Note that for the case $\ell_0=0$, the above inequality is true with the convention $\sum_{\emptyset}=0$ (since the inequality ${\cal W}_p^p(\mu, \nu)  \le\diamd^p$ is always true).

\smallskip \noindent {\sc Case 1} $(p>d)$. We first apply \eqref{eq:refwp} with $\ell_0=+\infty$ and obtain:
\begin{align}\label{eq:bound1case1}
{\cal W}_p^p(\mu, \nu) & \le\diamd^p \frac{2^p+1}{2}\sum_{\ell=1}^{\ell_0} 2^{(d-p)\ell} D^{\infty}(\mu,\nu)\le \diamd^p \frac{2^p+1}{2}\frac{2^{d-p}}{1-2^{d-p}} D^{\infty}(\mu,\nu).
\end{align}
Second, we apply \eqref{eq:refwp} with $\ell_0=\ell^\star-1$ where 
$$\ell^\star:=\inf\{\ell \ge 1, 2^{d \ell}D^{\infty}(\mu,\nu)\ge 2\}.$$
One can check that
\[
 \ell^\star=\bigg\lceil \frac{\log (2/D^{\infty}(\mu,\nu))}{d\log 2} \bigg\rceil\ge\lceil \tfrac 1d\rceil \ge  1,
\]
since $\log D^{\infty}(\mu,\nu)\le 0$. As a consequence, applying \eqref{eq:refwp} with $\ell_0= \ell^{\star}-1$, we get
\begin{align}
{\cal W}_p^p(\mu, \nu) & 
 \le \diamd^p \left[\frac{2^p+1}{2}\sum_{\ell=1}^{\ell^\star-1} 2^{(d-p)\ell} D^{\infty}(\mu,\nu) + 2^{-p(\ell^\star-1)}\right]\nonumber\\
\label{eq:techdisc}&  \le   \diamd^p \left[b_{p,d}(1-2^{-(p-d)(\ell^\star-1)}) D^{\infty}(\mu,\nu)+ 2^{-p(\ell^\star-1)}\right],
\end{align}
with 
$$ 
b_{p,d}=\frac{2^p+1}{2({2^{p-d}-1})}.
$$
For a given $r>0$, one can check that
\begin{equation}\label{eq:boundellstar}
\begin{split}
&2^{-r(\ell^\star-1)}\le 2^{-\frac rd \frac{\log (2/D^{\infty}(\mu,\nu))}{\log 2}+r}= 2^{r(1-\frac 1d)}D^{\infty}(\mu,\nu)^{\frac rd},\\
&2^{-r(\ell^\star-1)}\ge 2^{-\frac rd}D^{\infty}(\mu,\nu)^{\frac rd}.
\end{split}
\end{equation}
Plugging these inequalities into \eqref{eq:techdisc}, this leads to:
\begin{equation*}
{\cal W}_p^p(\mu, \nu)\le b_{p,d} D^\infty(\mu,\nu)+\left(-b_{p,d} 2^{1-\frac{p}{d}}+ 2^{p(1-\frac{1}{d})}\right)D^\infty(\mu,\nu)^{\frac{p}{d}}.
\end{equation*}
One can check that 
$$b_{p,d} 2^{1-\frac{p}{d}}- 2^{p(1-\frac{1}{d})}=\frac{2^{p(1-\frac{1}{d})}}{2^{p-d}-1}\left(1+2^{-p}-2^{p-d}\right).$$ 
This provides the  second announced estimate.\\

\noindent {\sc Case 2} $(p=d)$. Here, \eqref{eq:refwp} again applied with $\ell_0=\ell^\star-1$ yields
\begin{equation*}
{\cal W}_p^p(\mu, \nu)  
\le\diamd^p \left(\frac{2^p+1}{2}(\ell^\star-1)D^{\infty}(\mu,\nu)+2^{-p(\ell^\star-1)}\right).
\end{equation*}
Using that $\ell^\star-1<\frac{\log (2/D^{\infty}(\mu,\nu))}{d\log 2}$  and \eqref{eq:boundellstar} (applied with $r=p$), we obtain
\begin{equation*}
{\cal W}_p^p(\mu, \nu)  
\le\diamd^p \frac{2^p+1}{2}\frac{\log (2/D^{\infty}(\mu,\nu))}{d\log 2}D^{\infty}(\mu,\nu)+2^{p(1-\frac 1d)}D^{\infty}(\mu,\nu).
\end{equation*}
The estimate follows.

\smallskip
\noindent {\sc Case 3} ($p<d$). By \eqref{eq:refwp} applied with $\ell_0=\ell^\star-1$, we obtain similarly to \eqref{eq:techdisc}
\begin{equation*}
{\cal W}_p^p(\mu, \nu)  
\le  \diamd^p \left(\frac{(2^p+1)2^{(d-p)(\ell^\star-1)}}{2(1-{2^{p-d}})} D^{\infty}(\mu,\nu)+ 2^{-p(\ell^\star-1)}\right).
\end{equation*}
By the second inequality of \eqref{eq:boundellstar} and the one below (applied with $r=d-p$),
$$ 2^{r(\ell^\star-1)}\le 2^{\frac rd \frac{\log (2/D^{\infty}(\mu,\nu))}{\log 2}}= 2^{\frac{r}{d}}D^{\infty}(\mu,\nu)^{-\frac rd},\quad r\ge0,
$$
 we  deduce that
 \begin{equation*}
{\cal W}_p^p(\mu, \nu)  
\le  \diamd^p \left(\frac{(2^p+1)2^{-\frac{p}{d}}}{1-{2^{p-d}}} D^{\infty}(\mu,\nu)^{\frac{p}{d}}+ 2^{p(1-\frac 1d)}D^{\infty}(\mu,\nu)^{\frac pd}\right).
\end{equation*}

\noindent {\sc Step~2} (\textit{General case}).  Here, we have to use ~\cref{prop:lemmefondaFG}$(b)$ and thus to consider  the elements of $\tilde{{\cal P}}_{\ell}$. These elements take the form $]\![x,y]\!]$ defined in~\cref{lem:tech1}$(c)$. Hence, from this result and from \eqref{eq:FG2} and \eqref{eq:FG3}, we deduce that \eqref{eq:refwp} still holds true with $\tilde{{\cal P}}_{\ell}$.
The sequel of the above proof being entirely based on this inequality, we deduce that the conclusions also hold true in the general case.
\end{proof}

When $p=d=1$, the above bounds are sub-optimal due to the following proposition (where the norm is the absolute value).
\begin{prop}[One dimensional  setting for ${\cal W}_1$]\label{prop:1DW1Disc} If $p=d=1$, then
\[
{\cal W}_1(\mu, \nu) \le D^\star(\mu,\nu) \le D^{\infty}(\mu, \nu).
\]
\end{prop}

\begin{proof} This relies on the Koksma-Hlawka inequality, which reads as follows in one dimension in  the version established in~\cite{BouLep} or \cite{GilPag2026}. For every function $f:[0,1]\to \R$ with finite variation in the measure sense, meaning that there is a signed measure $m_f$ on $\big([0,1], {\cal B}or([0,1])\big)$ such that $m_f(\{0\})=0$ and $f(x) = f(1) + m_f([0,1-x])$, one has
\[
\big| \mu(f)-\nu(f)\big|\le D^\star(\mu,\nu) |m_f|([0,1])
\]
where $|m_f|$ stands for  the total variation measure of $m_f$. 
In one dimension, a Lipschitz continuous function $f$ has finite variation in the above sense since it is $du$-$a.e.$ differentiable with a bounded derivative $f'$ satisfying
 $$
 f(x)= f(0) +\int_0^x f'(u)du = f(1)-\int_0^{1-x}f'(1-v)dv
 $$
 so that $m_f (du)= -f'(1-u)du$ and $|m_f| = |f'(1-u)|du$. Then $m_f(\{0\})=0$ and $|m_f|([0,1])\le \|f'\|_{L^\infty(du)}= [f]_{\rm Lip}$.
 Consequently for every Lipschitz function
 \[
 \big| \mu(f)-\nu(f)\big|\le [f]_{\rm Lip}D^\star(\mu,\nu).
 \]
The Monge-Kantorovich representation of the ${\cal W}_1$-distance 
\[
{\cal W}_1(\mu,\nu) =\sup_{[f]_{\rm Lip}\le 1} \int f(d\mu-d\nu)
\]
yields the announced result.
\end{proof}

This result suggests that  the $\log$-term  in the upper-bound obtained in~\cref{prop:WvsDisc} for the case $p=d$ is  possibly superfluous. At least such is the case when $d=1$.  Pro\"inov's Theorem in the following section also leads in favor of the same direction. An extension of this result to general $K$--$S$ distance based  on another method  is proposed in Section~\ref{subsec:3.1} .

In order to partially synthesize~\cref{prop:lemmefondaFG}, we derive  the following corollary.
\begin{cor}\label{cor:discempiric} If ($d\ge 2$ and $p\neq d$) or ($d=1$),  there exists a real constant $K_{p,d}$ depending on $p$, $d$ such that, for every $\mu$, $\nu \!\in {\cal P}([0,1]^d)$
\[
{\cal W}_p(\mu, \nu) \le K_{p,d}\sup_{x,y\,\in [0,1]^d} |x-y| \big(D^{\infty}(\mu,\nu)\big)^{\frac 1p\wedge \frac 1d}.
\]
\end{cor}
\noindent {\bf Toward a Law of Iterated Logarithm (Monte Carlo simulation).} {Let $(U_n)_{n\ge 1}$ be an {\it i.i.d.} sequence of uniformly distributed vectors on $[0,1]^d$. Then Chung's Law of Iterated Logarithm (see~\cite{Chung49, Kiefer61}) for the star discrepancy reads
\[
\varlimsup_n \sqrt{\frac{2n}{\log\log n}}D^\star(U_1, \ldots, U_n)= 1 \quad \P\mbox{-}a.s.
\]
Combining this  result with that of  Corollary~\ref{cor:discempiric} yields that, if ($d\ge 2$ and $p\neq d$) or ($d=1$),  then there exists a real constant $K_{p,d}$ only depending on $p$, $d$ such that, under the assumptions of this  corollary
\[
\varlimsup_n \bigg({\frac{2n}{\log\log n}}\bigg)^{\frac 12(\frac 1p\wedge \frac 1d)}{\cal W}_p\Big(\frac1n \sum_{k=1}^n \delta_{U_k}, {\cal U}([0,1]^d)\Big)\le  K_{p,d} \quad \P\mbox{-}a.s.
\]
where $K_{p,d}$ is a finite real constant from Corollary~\ref{cor:discempiric}.

\subsection{A refinement when $p=1$ and $d>1$}
 In view of the connection with Pro\"inov's Theorem recalled in~\cref{sec:proinov}, we propose a refined result for the ${\cal W}_1$-distance when the dimension $d$ is greater than $2$. By an optimization strategy on the choice of $\ell_0$ defined in~\cref{prop:lemmefondaFG}, we get the sharper upper-bound with an explicit smaller constant.
\begin{thm} \label{prop:refinement}Let $\mu$ and $\nu$ two probability measures on $([0,1]^d, {\cal B}or([0,1]^d),\lambda_d)$.  Then, for any integer $d\ge 2$,
\[
{\cal W}_1(\mu, \nu) \le \diamd   {2^{-\frac 1d}}   \left(\frac{3(d-1)}{2(1-2^{1-d})}\right)^{\frac{1}{d}}\frac{2d}{d-1} D^{\infty}(\mu,\nu)^{\frac 1d}.
\]
In particular, 
\[
{\cal W}_1(\mu, \nu) \le \diamd       \kappa_d D^{\star}(\mu,\nu)^{\frac 1d}
\quad \mbox{ with }\quad  
\kappa_d=2^{1-\frac 1d}\left(\frac{3(d-1)}{2(1-2^{1-d})}\right)^{\frac{1}{d}}\frac{2d}{d-1}.
\]
\end{thm}

\begin{rem} $\rhd$ One can check  that $\kappa_d\xrightarrow{d\rightarrow+\infty}4$. 

\noindent $\rhd$ The optimization {of $\ell_0$ proposed in the proof below} is not completely natural in view of the proof of~\cref{prop:WvsDisc} where the integer $\ell^\star$ is precisely defined to optimize the bounds. However, the definition of $\ell^\star$ involves an upper integer part which may have bad effects on the constants.

\noindent $\rhd$ Such a strategy may also be applied in the other cases which may slightly improve the results at the price of technicalities that we considered useless for the paper.
\end{rem}
\begin{proof} We only treat the case where $\mu((0,1]^d)=\nu(([0,1]^d)=1$. The extension to the general case can be done exactly as in the proof of~\cref{prop:WvsDisc}.  We start from~\eqref{eq:refwp} when $p=1$, namely
\begin{align*}
{\cal W}_1(\mu, \nu) & \le\diamd \left(\frac{3}{2}\,\sum_{\ell=1}^{\ell_0} 2^{(d-1)\ell}   D^{\infty}(\mu,\nu)+2^{-\ell_0}\right).
\end{align*}
We introduce a parameter $a\ge 1$ and define $\ell_0^{(a)}$ by
\[
\ell_0^{(a)}=\bigg\lceil \frac{\log (2/(aD^{\infty}(\mu,\nu)))}{d\log 2} \bigg\rceil-1.
\]
Note that $\ell_0^{(a)}=\ell^\star-1$ when $a=1$. In order to ensure that $\ell_0\ge 0$, we first assume that
\begin{equation}\label{eq:conditiona}
 \frac{2}{aD^{\infty}(\mu,\nu)}>1.
 \end{equation}
 In this case, following the strategy of the proof of~\cref{prop:WvsDisc} when $p<d$ with $\ell_0=\ell_0^{(a)}$ (instead of $\ell_0=\ell^\star-1$), we get
 \begin{equation*}
{\cal W}_1(\mu, \nu)  
\le  \diamd 2^{-\frac{1}{d}}\left(\frac{3}{1-{2^{1-d}}} a^{\frac{1}{d}-1}+ 2 a^{\frac{1}{d}}\right)D^{\infty}(\mu,\nu)^{\frac 1d}.
\end{equation*}
This suggests to minimize the function $h$ defined by:
$$ h(a)= \frac{3}{1-2^{1-d}} a^{\frac{1}{d}-1}+ 2 a^{\frac{1}{d}}.$$
One checks that this functions attains its minimum at the point 
$$a^\star=\frac{3(d-1)}{2(1-2^{1-d})}\quad \textnormal{and that}\quad h(a^\star)=\left(\frac{3(d-1)}{2(1-2^{1-d})}\right)^{\frac{1}{d}}\left(\frac{2}{d-1}+2\right).$$
The result follows if $a^\star$ satisfies the condition \eqref{eq:conditiona}. Otherwise, 
$$ D^\infty(\mu,\nu)\ge \frac{2}{a^\star}=\frac{4(1-2^{1-d})}{3(d-1)},$$
and using that ${\cal W}_1(\mu,\nu)\le \diamd$, we get
$$
{\cal W}_1(\mu,\nu)\le \diamd \le \diamd \left(\frac{3(d-1)}{4(1-2^{1-d})}\right)^{\frac{1}{d}}D^\infty(\mu,\nu)^\frac{1}{d}\le  \diamd 2^{-\frac{1}{d}} h(a^\star)D^\infty(\mu,\nu)^\frac{1}{d}.$$
The previous bound is thus still available in this case.
 \end{proof}
 
\subsection{Connections with QMC \& MC methods} \label{sec:proinov}  The discrepancy is mostly used in the theory of uniformly distributed sequences  and their applications to Quasi-Monte Carlo simulation ({\em QMC}).  A distribution $\nu$ being fixed on $[0,1]^d$ --~mostly the uniform distribution ${\cal U}([0,1]^d) $~-- it is commonly used to measure the  way the empirical measure induced by a $\big([0,1]^d\big)^n$-valued $n$-tuple $(\xi_k)_{k=1,\ldots,n}$ approximates the original measure $\nu$. To be more precise one considers
\[
D^{\star,\nu}\big((\xi_k)_{k=1,\ldots,n}\big) := D^{\star}\Big( \frac 1n \sum_{k=1}^n \delta_{\xi_k}, \nu\Big)
\] 
or its counterpart $D^{\infty,\nu}\big((\xi_k)_{k=1,\ldots,n}\big)$ defined accordingly w.r.t. $D^{\infty}$. For  an introduction to  {\em QMC} methods in Numerical Probability, we refer among others to~\cite{Nied1992} or~\cite{GilPag2026}. One important theoretical result in this field is Pro\"inov's theorem (see~\cite{proinov1988}) which can be formulated as follows.
\begin{prop}[Pro\"inov Theorem (1984)] Let $d$, $n \!\in \N$ and let   $\big([0,1]^d\big)^n$-valued $n$-tuple $(\xi_k)_{k=1,\ldots,n}$. Then there exists a real constant $C_d\!\in [1,4]$ such that 
\[
{{\cal W}_1^{\ell^{\infty}}\Big( \frac 1n \sum_{k=1}^n \delta_{\xi_k}, {\cal U}([0,1]^d) \Big)\le C_d D^{\star, {\cal U}([0,1]^d)}\big((\xi_k)_{k=1,\ldots,n}\big)^{\frac 1d},}
\]
where ${\cal W}_1^{\ell^{\infty}}$ denotes the  Wasserstein distance w.r.t. the $\ell^\infty$-norm on $[0,1]^d$. Moreover when $d=1$, then $C_1=1$ and both error moduli attain their minimum, $n$ being fixed, at $\Big(\frac{2k-1}{2n}\Big)_{k=1,\ldots,n}$ with a common resulting value $\frac{1}{2n}$.
\end{prop}

\begin{rem}\label{rem:universality} $\bullet$ The definition of the star discrepancy in~\cite{proinov1988} is slightly different from that of $D^\star (\mu,\nu)$, namely $\sup_{x,\, y\in [0,1]^d} \big| \mu\big([\![0,y[\![\big) -  \nu\big([\![0,x[\![\big)\big|$,  but  this modulus turns out to be lower or equal to $D^\star (\mu,\nu)$ using arguments similar to those used to prove Lemma~\ref{lem:tech1}.

\smallskip 
\noindent  $\bullet$  Since $D^\infty\le 2^d D^\star$, we can compare the improved general constant $\kappa_d$ from~\cref{prop:refinement} the  (bounds known on) constant  $C_d$ appearing in Pro\"inov's Theorem (which holds in a more restricted framework),  having in mind that  ${\diamd =1}$ for the $\ell^{\infty}$-norm. Let us recall that this constant $\kappa_d$  is given by 
\begin{align*}
\kappa_d=2^{1-\frac 1d}\left(\frac{3(d-1)}{2(1-2^{1-d})}\right)^{\frac{1}{d}}\frac{2d}{d-1} \longrightarrow 4 \quad \mbox{ as }\quad d\to +\infty. 
\end{align*}
Numerical computations for medium values of $d$ are as follows:  if  $d=2$,  $\kappa_2\simeq 9.7980$,   if $d =3$,  $\kappa_3\simeq 7.5595$, if $d=4$,  $\kappa_4\simeq     6.7537$,     if $d=5$,  $\kappa_5\simeq 6.3096$, if $d=6$,  $\kappa_6\simeq     6.0147$, if $d=7$,  $\kappa_7\simeq     5.7983$,      if $d=8$,  $\kappa_8\simeq 5.6299$,  if $d=9$,  $\kappa_9\simeq    5.4937$,  if $d=10$,  $\kappa_{10}\simeq    5.3806$,     if $d=11$,  $\kappa_{11}\simeq    5.2850$,  if $d=12$,  $\kappa_{12}\simeq        5.2028$,  if $d=20$,  $\kappa_{20} \simeq  4.8087$. if $d=50$,  $\kappa_{50} \simeq  4.3867$,  if $d=100$,  $\kappa_{100} \simeq 4.2182$.

\smallskip 
\noindent  $\bullet$ Our constants are thus slightly larger than those of the original theorem (which  lie into $[1,4]$) but it is worth noting that our bounds are universal: they do not hold only for the uniform distribution {${\cal U}([0,1]^d)$ but also for any distribution $\nu$ on $[0,1]^d$} (and {any}  empirical measure).
\end{rem}

\medskip
\noindent {\bf Toward a Law of Iterated Logarithm.}  Let $(U_n)_{n\ge 1}$ be an i.i.d. sequence of uniformly distributed vectors on $[0,1]^d$. Then \ Chung's Law of Iterated Logarithm (see~\cite{Chung49}) for the star discrepancy reads
\[
\varlimsup_n \sqrt{\frac{2n}{\log\log n}}D^\star(U_1, \ldots, U_n)= 1 \quad \P\mbox{-}a.s.
\]
Combining this results with that of  Corollary~\ref{cor:discempiric} yields that, if $d\ge 2$ and $p\neq d$ or ($d=1$),  then there exists a real constant $K_{p,d}$ only depending on $p$, $d$ such that, under the assumptions of this  corollary
\[
\varlimsup_n \bigg({\frac{2n}{\log\log n}}\bigg)^{\frac 12(\frac 1p\wedge \frac 1d)}{\cal W}_p\Big(\frac1n \sum_{k=1}^n \delta_{U_k}, {\cal U}([0,1]^d)\Big)\le  K_{p,d} \quad \P\mbox{-}a.s.
\]
where $K_{p,d}$ is a finite real constant from Corollary~\ref{cor:discempiric}.

\subsection{Bounding {the} star discrepancy  by  {the}  $L^1$-Wasserstein distance}
We refer to Section~\ref{subsec:KSvsW_1unbound} devoted to Kolmogorov-Smirnov distance between  distributions with possibly unbounded supports. Note that these bounds require that at least one of the two distributions under consideration is absolutely continuous. The obtained bound cannot be improved in the case of $[0,1]^d$-supported distributions, at least in  a reasonably general framework.

\section{Kolmogorov-Smirnov distance {\em vs}  {$p$-}Wasserstein distance on ${\cal P}_p(\R^d)$}\label{sec:3}

\subsection{Bounding the {$p$-}Wasserstein distance by the $K$--$S$ distance} \label{subsec:3.1}We consider now probability distributions on the whole space $\R^d$ and we straightforwardly update the definitions of the star and uniform   discrepancies. The first one is then also known as the {\em Kolmogorov-Smirnov distance} ($K$--$S$ distance). This section allows to treat the  $[0,1]^d $-supported distributions but with worse constants where the $K$--$S$ distance is commonly known as (star) discrepancy.

\begin{defn} Let $\mu$ and $\nu$ two probability measures on $([0,1]^d, {\cal B}or([0,1]^d),\lambda_d)$. We define the {\em Kolmogorov -Smirnov distance}, denoted $K$--$S$, by
\[
D^{\star}(\mu,\nu) =Ê\sup_{x\in \R^d} \big| \mu\big(]\!]-\infty,x]\!]\big) -  \nu\big((]\!]-\infty,x]\!]\big)\big|,
\]
where, by an abuse of notation, we also denote $-\mathbf{\infty} = (-\infty, \ldots, -\infty)$. This distance can be simply seen as  {\em star discrepancy}  defined in a more general setting. We also define the {\em uniform discrepancy} between $\mu$ and $\nu$ by
\[
D^{\infty}(\mu,\nu) =Ê\sup_{x,\, y\in \R^d} \big| \mu\big([\![x,y]\!]\big) -  \nu\big([\![x,y]\!]\big)\big|.
\]
\end{defn}
One easily checks like in Lemma~\ref{lem:tech1} that, with these definitions,
\begin{equation}\label{eq:alterdicunif}
D^{\infty}(\mu,\nu) =Ê\sup_{x,y \in \R^d} \big| \mu\big(]\!]x,y]\!]\big) -  \nu\big(]\!]x,y]\!]\big)\big|
\end{equation}
since $\mu ([\![x,y]\!])= \lim_n\mu (]\!]x-\mbox{\bf1}/n,y]\!])$ and that  the bounds~\eqref{eq:discbound} 
\begin{equation}\label{eq:dinftyvsdstar}
D^\star\le D^\infty\le 2^d D^\star
\end{equation}
between these quantities still hold.

The following Proposition, which is the combination of Lemmas~5 and~6 from~\cite{FG}, is the key result on which we rely in this {section}. {We set:}
$$ {\cal B}_n=(-2^{-n},2^{-n}]^d\backslash (-2^{-(n-1)},2^{-(n-1)}]^d.$$
}
\begin{prop} \label{prop:upperWunbounded}Let $p\!\in (0, +\infty)$ and let $d\ge1$.  There exists a positive constant $K_{p,d}$ such that for every pair   $(\mu,\nu)\in{\cal P}_p(\R^d)^2$,
\begin{equation}\label{eq:W-major}
{\cal W}_p^p(\mu,\nu) \le K_{p,d} \sum_{n\ge 0}2^{pn} \sum_{\ell\ge 0}2^{-p\ell}\sum_{F\in {\cal P}_{\ell}}\big| \mu(2^nF\cap \B_n)Ê- \nu(2^nF\cap \B_n)Ê\big|,
\end{equation}
where $2^nF = \{2^nx, \; x\!\in F\}$, ${\cal P}_0= \{ (-1,1]^d\}$ and, for every $\ell\ge 1$,  
$$
{\cal P}_\ell= \Big\{a+(-2^{-\ell}, 2^{-\ell}]^d,\; a = \frac{2{\bf k}+\mbox{\bf 1}}{2^{\ell}},\,{\bf k}\in\{-2^{\ell-1},\ldots-1,0,1,\ldots, 2^{\ell-1}-1\}^d  \Big\}.
$$
 Note that   ${\rm card}({\cal P}_\ell)= 2^{d\ell}$ and (with obvious notation)  ${\rm card}(2^n{\cal P}_\ell\cap  {\cal B}_n)= 2^{d(\ell-1)^+}$.
\end{prop}

\begin{thm}\label{prop:vitessesur Rd}Let $q>p$ and $\mu$, $\nu$ two probability measures with finite $q$-moments. 
There exist real constants $\kappa_{p,q,d}>0$ such that 
\begin{equation}\label{eq:boundWpbyKS}
{\cal W}_p^p (\mu,\nu)\le \kappa_{p,q,d}\big(M_{\frac{\mu+\nu}{2}}(q) \vee 1)\left\{\begin{array}{ll}
 D^\infty(\mu,\nu)^{\frac pd}  &\mbox{ if }p<\frac{dq}{q+d}\\
 D^\infty(\mu,\nu)^{1-\frac pq}  &\mbox{ if } \frac{dq}{q+d}<p \mbox{ and }p\neq d
 \end{array}
 \right.
\end{equation}
where $M_{\frac{\mu+\nu}{2}}(q)=\frac 12 \int_{\R^d} |\xi|^q (\mu+\nu)(d\xi)$. 
\end{thm}

\noindent {\bf Remarks.}  $\bullet$ The above result can be partially summed up into: if $p\neq d$ and $p\neq \frac{dq}{d+q}$ then
\[
{\cal W}_p^p (\mu,\nu)\le \kappa_{p,q,d} \big(M_{\frac{\mu+\nu}{2}}(q) \vee 1)D^\infty(\mu,\nu)^{\frac pd\wedge (1-\frac pq)}.
\]
This is in line with what was obtained for $[0,1]^d$-supported distributions (for which $q=+\infty$).

\smallskip
\noindent $\bullet$ If $p=d$ our approach fails to provide a direct bound. However , for every $\varepsilon\!\in (0, q-d)$, one has 
\begin{equation}\label{eq:boundWpbyKSL1}
{\cal W}_p(\mu,\nu)\le {\cal W}_{d+\varepsilon}(\mu,\nu)\le \kappa_{1+\ve,q,d}^{\frac{1}{d+\ve}} \big(M_{\frac{\mu+\nu}{2}}(q) \vee 1)^{\frac{1}{d+\ve}}D^\infty(\mu,\nu)^{\frac{1}{d+\ve}-\frac1q}.
\end{equation}

\noindent $\bullet$ However, in one dimension,  a specific approach is possible based on the representation formula (see \emph{e.g.} \cite{villani_topics})
$$
{\cal W}_1(\mu,\nu)=\int_{-\infty}^{\infty} |F_\mu(x)-F_\nu(x)| dx,
$$
where $F_\mu$ and $F_\nu$ denote the cumulative distribution functions of $\mu$ and $\nu$ respectively.
Then, for every $a>0$, 
\begin{align}
\nonumber {\cal W}_1(\mu,\nu)&\le \int_{-\infty}^{-a} F_\mu(x)+F_\nu(x) dx+ 2 a D^\star(\mu,\nu)+ \int_{a}^{+\infty} (1-F_\mu(x))+(1-F_\nu(x)) dx\\
\label{eq:MAjoW1}&\le \int_a^{+\infty} \PE(|X|>x)+\PE(|Y|>x) dx+2 a D^\star(\mu,\nu)
\end{align}
where $X$ and $Y$   {are $\mu$ and $\nu$}-distributed respectively. If $X$ and $Y$ both have finite moments of order  $q$, one has 
$$
\int_a^{+\infty} \PE(|X|>x) dx\le \ES[|X|^q] \int_a^{+\infty} x^{-q} dx \le \ES[|X|^q] a^{1-q},
$$
which yields after an obvious optimization
$$ 
{\cal W}_1(\mu,\nu)\le 2M_{\frac{\mu+\nu}{2}}(q) D^\star(\mu,\nu)^{1-\frac{1}{q}}.
$$

 \noindent $\bullet$ If $\mu$ and $\nu $ have exponential moments  in the sense that $\int_{\R} e^{\lambda|\xi|} (\mu+\nu)(d\xi)<+\infty$ for some $\lambda>0$, then it follows from~\eqref{eq:MAjoW1}  that,
 \[
{\cal W}_1(\mu,\nu)\le \int_{\R} e^{\lambda|\xi|} (\mu+\nu)(d\xi) \frac{e^{-\lambda a}}{\lambda} +2 a D^\star(\mu,\nu)
\]
Setting $a= \frac{1}{\lambda} \log(1/D^\infty(\mu,\nu))$ yields 
\[
{\cal W}_1(\mu,\nu)\le \frac{1}{\lambda} \bigg(\int_{\R} e^{\lambda|\xi|} (\mu+\nu)(d\xi) D^\star(\mu,\nu)+ 2  D^\star(\mu,\nu)\log(1/D^\infty(\mu,\nu))\bigg).
\]
\noindent $\bullet$ If $\mu_n $, $n\ge 1$ are   (uniform) empirical measures of the form $\mu = \mu_n = \frac1n \sum_{k=1}^n \delta_{\xi^n_k}$ here $\xi^n_1,\ldots, \xi^n _n\!\in \R^d$. Then 
$$
M_{\mu_n}(q),\; n\ge 1, \mbox{ is bounded  iff  } \sup_{n\ge 1} \frac 1n \sum_{k=1}^n |\xi^n_k|^q  <+\infty.
$$

\noindent {\em Proof.}  {\sc Step~1}.  Let $\ell\ge 1$. It is clear that $\frac{2k+\mbox{\bf 1}}{2^{\ell}}+(-2^{-\ell}, 2^{-\ell}]^d=]\!]k 2^{-\ell}, (k+\mbox{\bf 1})2^{-\ell} ]\!]$  is a semi-open box so that, for every $F\!\in {\cal P}_\ell$, either $2^nF \cap  {\cal B}_n = \varnothing$  for $2^{d(\ell-1)}$ semi-open boxes or, for the $(2^d-1)2^{d(\ell-1)}= 2^{d\ell}(1-2^{-d})$  others
\[
\big| \mu(2^nF\cap \B_n)Ê- \nu(2^nF\cap \B_n)Ê\big|\le D^\infty(\mu,\nu)
\]
owing to~\eqref{eq:alterdicunif}.

If $\ell=0$,   $\big| \mu(\B_n)Ê- \nu(\B_n)Ê\big|\le \min(\mu(\B_n)Ê+ \nu(\B_n), D^\infty(\mu,\nu)\big)$.

Consequently, if we set 
$$
m= \frac 12 (\mu+\nu) \quad\mbox{ and } \quad M_m(q) = \int_{\R^d} |\xi|_\infty^q m(\d\xi)= \frac 12\Big(  \int_{\R^d} |\xi|_\infty^q\mu(\d\xi)+ \int_{\R^d} |\xi|_\infty^q\nu(\d\xi)\Big),$$ 
one has 
\begin{align}
\nonumber \sum_{F\in {\cal P}_\ell\cap  {\cal B}_n} \big| \mu(2^nF\cap \B_n)Ê- \nu(2^nF\cap \B_n)Ê\big|&\le \min \big( \mu( {\cal B}_n)+\nu( {\cal B}_n), 2^{d(\ell-1)^+}D^\infty(\mu,\nu)\big)\\
\nonumber &  \le \min \big( 2 m( {\cal B}_n), 2^{d\ell}D^\infty(\mu,\nu)\big)\\
 \label{eq:ineqfonda} & \le  \min \big( 2M_m(q)2^{-q(n-1)}, 2^{d\ell}D^\infty(\mu,\nu)\big),
\end{align}
where we used the triangle inequality to establish the left  bound in the min of the first line. 

\smallskip  
\noindent {\sc Step~2} ({\em Technical lemma}).
\begin{lem}\label{lem:tech}
Let $p>0$. Let $t>0$ be fixed and let $L: (0, +\infty)\to \R_+$ be defined by 
\[
L(u):=\sum_{\ell\ge 0} 2^{-p\ell} \min\big(u, 2^{d\ell}D^\infty(\mu,\nu)\big).
\]

The function $L$ satisfies the following upper-bounds depending on $p$ and the dimension $d$ where $C_{p,d}>0$ denotes    a positive constant only depending on $p$, $\b$, $d$ that may vary from line to line.
\begin{itemize}
\item If $p>d $, then 
\[
L(u) \le  C_{p,d}\min\Big(u,D^\infty(\mu,\nu)\Big).
\]
\item  If $p=d$ then,
	\[
	L(u) \le C_{p,d} (1+\left(\log(u/D^\infty(\mu,\nu))\right)_{{+}})D^\infty(\mu,\nu).
	\]
%
\item If $p < d $, then
$$
L(u) \le C_{p,d} \Big(u^{1-\frac pd}D^\infty(\mu,\nu)^{\frac pd}\mbox{\bf 1}_{\{u> D^\infty(\mu,\nu)\}}+ u\mbox{\bf 1}_{\{u\le D^\infty(\mu,\nu)\}}\Big).
$$
\end{itemize}
\end{lem}

\begin{proof}  $\triangleright$ If $p>d$, one has
\[
L(u) \le \min\bigg(u\sum_{\ell\ge 0} 2^{-p\ell},D^\infty(\mu,\nu) \sum_{\ell\ge 0} 2^{-(p-d)} \bigg) = \min\Big( \frac{u}{1-2^{-p}}, \frac{D^\infty(\mu,\nu)}{1-2^{-(p-d)}}\Big).
\]
 $\triangleright$  If $p=d$  there are two sub-cases. If $u<2^{-d}D^\infty(\mu,\nu)$, then $L(u) = \frac{u}{1-2^{-p}}$. Otherwise $\ell^*= \left\lceil \frac{\log(u/D^\infty(\mu,\nu))}{d\log 2}\right \rceil\!\in\Big [\frac{\log(u/D^\infty(\mu,\nu))}{d\log 2},  1+ \frac{\log(u/D^\infty(\mu,\nu)}{d\log 2}\Big)$   so that  $\ell^*\ge 0$ and 
\[
L(u) \le \ell^* D^\infty(\mu,\nu) + \frac{2^{-p\ell^*}}{1-2^{-p}} \le \Big(1+\frac{1}{1-2^{-p}} +\frac{\log(u/D^\infty(\mu,\nu))}{d\log 2})\Big) D^\infty(\mu,\nu).
\]

\noindent   $\triangleright$ If $p < d $ either  $u< 2^{-d}D^\infty(\mu,\nu)$ and  $L(u) = \frac{u}{1-2^{-p}}$. Otherwise $\ell^*$ defined as above  is nonnegative and
\begin{align*} 
L(u)  & =\sum_{\ell=0}^{\ell^*-1} 2^{-p\ell} 2^{d\ell}D^{\infty}(\mu,\nu) + \sum_{\ell\ge \ell^*} 2^{-p\ell}u\\
	&  =  D^{\infty}(\mu,\nu) \frac{2^{(d-p)\ell^*}-1}{2^{d-p}-1}+ u \frac{2^{-p\ell^*}}{1-2^{-p}}\\
	& \le  D^{\infty}(\mu,\nu) \frac{2^{(d-p)(1+\log(u/D^\infty(\mu,\nu))} }{2^{d-p}-1} + u\, 2^{-p} \frac{2^{-p\frac{\log(u/D^\infty(\mu,\nu))}{d\log 2}}}{1-2^{-p}}\\
	&\le \Big( \frac{1}{1-2^{-(d-p)}} + \frac{1}{2^p-1} \Big)\big(D^\infty(\mu,\nu)\big)^{\frac pd}u^{1-\frac pd}.
\end{align*}
\end{proof}

\noindent {\sc Step~3}. It follows from \eqref{eq:W-major}, Step~1 and the definition of the function $L$ that
\begin{align*}
{\cal W}_p^p(\mu,\nu)& \le K_{p,d} \sum_{n\ge 0}2^{pn} \sum_{\ell\ge 0}2^{-p\ell}\sum_{F\in {\cal P}_{\ell}}\big| \mu(2^nF\cap \B_n)Ê- \nu(2^nF\cap \B_n)Ê\big|\\
& \le K_{p,d} \sum_{n\ge 0}2^{pn} \sum_{\ell\ge 0}2^{-p\ell}\sum_{F\in {\cal P}_{\ell}} \min \big( 2^{1+q}M_m(q)2^{-qn}, 2^{d\ell}D^\infty(\mu,\nu)\big)\\
&\le K'_{p,d}(M_m(q)\vee1) \sum_{n\ge 0}2^{pn}\sum_{\ell\ge 0}2^{-p\ell}\sum_{F\in {\cal P}_{\ell}} \min \big(2^{-qn}, 2^{d\ell}D^\infty(\mu,\nu)\big)\\
& = K'_{p,d}(M_m(q)\vee1) \sum_{n\ge 0}2^{pn}L\big(2^{-qn},D^\infty(\mu,\nu)\big).
\end{align*}
Now we inspect the usual three cases. The  letters $C$ and $c$ denote positive constants depending only on its indices that may vary from line to line.

\smallskip
\noindent  $\triangleright$ If $p>d$, it follows from Proposition~\ref{prop:upperWunbounded} and Lemma~\ref{lem:tech} that 
\begin{align*} 
{\cal W}_p^p(\mu,\nu) &\le C'_{p,q,d}\sum_{n\ge 0}2^{np} \min\big(2^{-qn}, D^\infty(\mu,\nu)\big)\\
&  =  C'_{p,q,d} \Big(n^\star D^\infty(\mu,\nu)+\sum_{n\ge n^\star} 2^{n(p-q)}\Big)
\end{align*}
where $n^\star = \left\lceil -\frac{\log(D^\infty(\mu,\nu))}{q\log 2} \right\rceil\ge 1$.
Hence
\begin{align*} 
{\cal W}_p^p(\mu,\nu) &\le C_{p,q,d}\bigg(\frac{2^{-n^\star(q-p)}}{1-2^{-(q-p)}} +\frac{1}{q\log2}\log(1/D^\infty(\mu,\nu))D^\infty(\mu,\nu)\bigg)\\
				  &\le C_{p,q,d}\bigg(\frac{(D^\infty(\mu,\nu))^{1-\frac pq}}{1-2^{-(q-p)}} +\frac{1}{q\log2}\log(1/D^\infty(\mu,\nu))D^\infty(\mu,\nu)\bigg)\\
				  & \le C_{p,q,d} D^\infty(\mu,\nu)^{1-\frac pq}.
\end{align*}

\noindent  $\triangleright$ If $p<\frac{qd}{q+d}$  then it follows Proposition~\ref{prop:upperWunbounded} and Lemma~\ref{lem:tech} that 
\begin{align*} 
{\cal W}_p^p(\mu,\nu) &\le C_{p,q,d} \sum_{n\ge 0} 2^{(p-q(1-\frac pd))n} D^\infty(\mu,\nu)^{\frac pd} \mbox{\bf 1}_{\{c_{p,q,d}2^{-qn}>D^\infty(\mu,\nu)\}}+ 2^{-(q-p)n} \mbox{\bf 1}_{\{c_{p,q,d}2^{-qn}\le D^\infty(\mu,\nu)\}}\\
& = C_{p,q,d}\bigg(\sum_{n=0}^{n^\star-1}  2^{(p-q(1-\frac pd))n} D^\infty(\mu,\nu)^{\frac pd} + \frac{2^{-(q-p)n^\star}}{1-2^{-(q-p)}}\bigg)\\
&= C_{p,q,d} \bigg(\frac{2^{(p-q(1-\frac pd))n^\star}-1}{2^{(p-q(1-\frac pd)}-1} D^\infty(\mu,\nu)^{\frac pd}  +2^{-(1-\frac pd)d} (D^\infty(\mu,\nu))^{1-\frac pq} \bigg)
\end{align*}
with $n^\star= \left\lceil \frac{\log(c_{pq,d}/D^\infty(\mu,\nu))}{q\log 2}\right\rceil$. Hence, if $p< \frac{qd}{q+d}$ then $p-q(1-\frac pd)<0$ so that 
\[
{\cal W}_p^p(\mu,\nu) \le C_{p,q,d}  \big(D^\infty(\mu,\nu)^{\frac pd} +  D^\infty(\mu,\nu)^{1-\frac pq} \big)\le C^{(3)}_{p,q,d}D^\infty(\mu,\nu))^{\frac pd} 
\] 
since $\frac pd \le 1-\frac pd$ and $D^\infty(\mu,\nu) \le 1$. 

If $p> \frac{qd}{q+d}$ then $p-q(1-\frac pd)>0$ and one easily checks using that $n^\star < 1+\frac{\log(c_{pq,d}/D^\infty(\mu,\nu))}{q\log 2}$ that
\[
{\cal W}_p^p(\mu,\nu) \le C_{p,q,d}   D^\infty(\mu,\nu)^{1-\frac pq} .
\]
\hfill $\Box$

\subsection{Bounding the $K$--$S$-distance by {the} $L^1$-Wasserstein distance}\label{subsec:KSvsW_1unbound}
Let us denote by $\|u\|_{\infty}= \max_{i=1,\ldots,d} |u^i|$ the $\ell^\infty$-norm and let us define, for every $A \subset \R^d$ and $x\!\in \R^d$, $d_\infty(x,A) =\inf _{a\in A} |x-a|_{\infty}$. The key property of this section is the Monge-Kantorovich representation of the $L^1$- Wasserstein distance, namely
\begin{equation}\label{eq:MKrep}
{\cal W}^{\ell^{\infty}}_1(\mu, \nu)= \sup \bigg\{\int fd\mu -\int f d\nu, f\!\in {\rm Lip}(\R^d, \R),\, \mbox{ with } [f]_{\rm Lip}\le 1\bigg\},
\end{equation}
where $[f]_{\rm Lip}:=\sup_{x,y\in \R^d}\frac{|f(x)-f(y)|}{|x-y|_{_\infty}}$.

{Having in mind that the topology induced by the $K$--$S$ distance is{ finer}  than that induced by  ${\cal W}_1$, as emphasized by the counterexample (see \eqref{eq:counterexampleb}),  we need an additional assumption on one of the two probability measures under consideration to  bound  the first distance  by the second one. Thus, we will assume  in the proposition below that at least one of the two measures is absolutely continuous with respect to the Lebesgue measure.}

\begin{thm}[Bounding  star discrepancy discrepancy by $L^1$-Wasserstein distance]\label{thm:boundDbyW} Let $\nu $ be an absolutely  continuous distribution on $\R^d$
with density $g\!\in {\cal L}^{\frac{r}{r-1}}(\R^d, \lambda_d)$ for some $r\!\in (1,+\infty]$ and finite first moment. Then, for every probability distribution  $\mu$ on $\R^d$with  finite first moment,
\begin{equation}\label{eq:DstarleW1}
D^\star(\mu,\nu)\le C_{r,d}{\cal W}^{\ell^\infty}_1(\mu,\nu)^{\frac{d}{r+d}} \big \| g\big\|^{\frac{r}{r+d}}_{{\cal L}^{\frac{r}{r-1}}(\lambda_d)}.
\end{equation}
with $C_{r,d} =  \begin{pmatrix}r+d \\ r \end{pmatrix}^{-\frac{1}{r+d}}\bigg( \Big(\frac dr\Big)^{\frac{r}{r+d}} + \Big(\frac rd\Big)^{\frac{d}{r+d}} \bigg)$ for $r\ge 1$.
If $g$ is bounded, one has:
\[
D^\star(\mu,\nu)\le C_{1,d}{\cal W}^{\ell^\infty}_1(\mu,\nu)^{\frac{d}{d+1}} \big\| g\big\|^{\frac{1}{d+1}}_{\infty}.
\]
\end{thm}

\noindent {\bf Remark.} Note the  case of a bounded density corresponds to $r=1$ and is consistent with the general formula for $r>1$. Also note  that this constant  goes to $1$ as $d\to \infty$.

\begin{proof} Let $x\!\in [0,1]^d$,  let $f_x= \mbox{\bf 1}_{]\!]-\infty,x]\!]}$ and, for every $\ve >0$, $f_{x,\ve} = \Big(1-\frac{d_\infty(\cdot, ]\!]-\infty,x]\!])}{\ve} \Big)^+$ and $\tilde f_{x,\ve} = \Big(1-\frac{d_\infty(\cdot, ]\!]-\infty,x-\ve \mathbf{1}]\!])}{\ve} \Big)^+$ (with the convention on boxes).  The functions $f_{x,\ve}$
and  $\tilde f_{x,\ve} $ are clearly $\frac{1}{\ve}$-Lipschitz for the $\ell^\infty$-norm. It is clear that $\tilde f_{x,\ve } \le  f_x\le f_{x,\ve}$.

Let us compute $d_\infty(u,]\!]-\infty,x]\!])$ for every $u\!\in \R^d\setminus ]\!]-\infty,x]\!]$. First, note that he continuous  convex function $\varphi_u:y\mapsto |u-y|_\infty$  attains its minimum on the boundary of  the box $]\!]-\infty,x]\!]$ namely
\[
\partial ]\!]-\infty,x]\!]= \bigcup_{1\le i\le d}\prod_{1\le j \le i-1} \{x^j\}\times (-\infty, x^i]\times \prod_{i+1\le j \le d} \{x^j\}.
\]
First note that $\inf_{y\in ]\!]-\infty, x]\!]}\varphi_i(y)  = \min _{y\in B_{\ell^\infty}(x, 2|u-x]_{_\infty})}|u-y|_{_\infty}$ hence ${\rm argmin} \, \varphi_u$ is nonempty. If not included in  $\partial ]\!]-\infty,x]\!$,  let $y^*\in {\rm argmin}\,\varphi_u\cap ]\!]-\infty,x[\![$  and let $g(t)= \varphi_u(tu+(1-t)y^*$), $t\in [0,1]$. The function $g$ is nonnegative, continuous and  convex and  there exists $\eta>0$ such that $u+(1-t)y^*\!\in  ]\!]-\infty,x[\![$ for $t\!\in (0,\eta]$. Hence the right derivative $g'_r(0)\ge 0$ by definition of $y^*$. As $g$ is convex, it is also  non-decreasing. Noting that $g(1)= 0$ implies that $g$ is identically $0$. Then $g(0)=0$ so that  $u= y^*$ which is impossible since $u$ does not belong  to $]\!]-\infty,x]\!]$. 
Hence, one easily checks that  
$$
d_\infty(u,]\!]-\infty,x]\!]) = \min_{i=1,\ldots,d} (x ^i-u^i)^+.
$$
Then 
\begin{align}
\nonumber \mu(]\!]-\infty,x]\!])-\nu(]\!]-\infty,x]\!]) &= \int \underbrace{(f_x-f_{x,\ve} )}_{\le 0}d\mu + \int f_{x,Ê\ve}d\mu -\int f_{x,Ê\ve}d\nu + \int (f_{x,\ve} -f_x)d\nu\\
\label{eq:ineqmunu1} & \le \frac{1}{\ve}{\cal W}_1^{\ell_\infty}(\mu,\nu)  + \int (f_{x,\ve} -f_x)gd\lambda_d
\end{align}
owing to~\eqref{eq:MKrep}.  Now it follows from the expression for $d_\infty(u,]\!]-\infty,x]\!])$ that, for every $u\!\in [0,1]^d$,
\begin{align*}
0\le f_{x,\ve}(u) -f_x(u) &= \bigg(1- \frac{d_{\infty}(u, ]\!]-\infty,x]\!])}{\ve} \bigg)^+ 
\mbox{\bf 1}_{]\!]-\infty,x]\!]^c}(u)\\
&=  \bigg(1- \frac{d_\infty(u, ]\!]-\infty,x]\!])}{\ve} \bigg) \mbox{\bf 1}_{[\![x,x+\ve \mathbf{1}]\!]^c}(u)\\
&\le  \bigg(1- \frac{\max_{i=1,\ldots,d}(u^i -x^i)}{\ve} \bigg) \mbox{\bf 1}_{x^i\le u^i\le x^i +\ve , i=1\ldots d}.
\end{align*}
Assume $r\!\in (1, +\infty)$. It follows from H\"older inequality that 
\begin{equation}\label{eq:majorfx}
\int (f_{x,\ve} -f_x)d\nu\le \bigg(\int_{ \prod_i[x^i, x^i+\ve]} \bigg[\Big(1- \frac{\max_{i=1,\ldots,d}(u^i -x^i)}{\ve} \Big)\bigg]^rdu\bigg)^{\frac 1r} \big\|g\big\|_{{\cal L}^{\frac{r}{r-1}}(\lambda_d)} .
\end{equation}
Now 
\begin{align*}
\int_{[0,1]^d \cap \prod_i[x^i, x^i+\ve]} \bigg[\Big(1- \frac{\max_{i=1,\ldots,d}(u^i -x^i)}{\ve} \Big)\bigg]^rdu &\le \int_{ \prod_i[x^i, x^i+\ve]} \bigg[\Big(1- \frac{\max_{i=1,\ldots,d}(u^i -x^i)}{\ve} \Big)\bigg]^rdu\\
&=\ve^d \int_{[0,1]^d} \bigg[\Big(1- \max_{i=1,\ldots,d}v^i \Big)\bigg]^rdv\\
 & = \ve^d 	\int_{[0,1]^d} (\min_{i=1,\ldots,d}w^i)^rdw \\
 &= \ve^d d!\int_{0<w^1<\cdots<w^d<1} (w^1)^rdw\\
&=\ve^d d!\frac{r!}{(r+d)!}	= 	\ve^d\begin{pmatrix}d+r\\ r	\end{pmatrix}^{-1}.														 
\end{align*}
Inserting this in~\eqref{eq:majorfx} and then in~\eqref{eq:ineqmunu1} yields
\[
 \mu(]\!]-\infty,x]\!])-\nu(]\!]-\infty,x]\!]) \le \frac{1}{\ve}{\cal W}^{\ell^\infty}_1(\mu,\nu)  +  \ve^{\frac dr}\begin{pmatrix}d+r\\ r	\end{pmatrix}^{-\frac 1r}\big\|g\big\|_{{\cal L}^{\frac{r}{r-1}}(\lambda_d)}.
\]
one shows likewise using $\tilde f_{x,\ve}$ that $ \nu(]\!]-\infty,x]\!])-\mu(]\!]-\infty,x]\!])$ satisfies  the same inequality since $\tilde f_{x,\ve}-f_x\le 0$.  Consequently, for every $\ve  >0$,
\[
D^\star(\mu,\nu)\le  \frac{1}{\ve}{\cal W}^{\ell^\infty}_1(\mu,\nu)  +  \ve^{\frac dr}C_{r,d}\big\|g\big\|_{{\cal L}^{\frac{r}{r-1}}(\lambda_d)}.
\]
with $C_{r,d}=\begin{pmatrix}d+r\\ r	\end{pmatrix}^{-\frac 1r}$.
One concludes by setting $\displaystyle \ve = \Big(\frac{{\cal W}^{\ell^\infty}_1(\mu,\nu)}{\|g\|_{{\cal L}^{\frac{r}{r-1}}(\lambda_d)}}\frac{r}{dC_{r,d}} \Big)^{\frac{r}{d+r}}$ at which the above function of $\ve$ attains its minimum. The case $r=1$ follows likewise.
 \end{proof}

\noindent {\bf Remarks.} {$\bullet$ Assume $d\ge 2$. Based on~\eqref{eq:boundWpbyKS} and~\eqref{eq:DstarleW1} with  $\nu= g\cdot\lambda_d$,  with $g\!\in {\cal L}^{\frac{r}{r-1}}(\lambda_d)$ for some $r\ge 1$, one easily deduces that for $q>\frac{d}{d-1}$ and every $\ve >0$ small enough
\[
\Big( M_{\frac{\mu+\nu}{2}}(q) \vee 1\Big)^{-1} {\cal W}_1(\mu,\nu)  \preceq D^\star(\mu,\nu) ^{\frac 1d}\preceq \big \| g\big\|^{\frac{r}{d(r+d)}}_{{\cal L}^{\frac{r}{r-1}}(\lambda_d)} {\cal W}_1(\mu,\nu)^{\frac{1}{r+d}},
\] 
where $\preceq$ stands for ``lower up to a constant'' (possibly depending on $r$, $q$, $d$).  This  suggests that this upper-bound is not sharp, having in mind that if $\mu_n= \frac 1n \sum_{k=1}^n \delta_{\frac{2k-1}{2n}}$ and $\nu = {\cal U}([0,1])$, then, for every $n\ge 1$, 
\[
{\cal W}_1(\mu_n ,\nu)= D^{\star}(\mu_n,\nu)= \frac{1}{2n}.
\] 
}

\noindent $\bullet$ In~\cite{GauLi}, a similar upper bound is established for more general  ``smooth Wasserstein'' distances  $d_m$ defined {by} $ d_m(\mu,\nu)=\sup_{f\in {\cal H}_m}|\mu(f)-\nu(f)|$
where ${\cal H}_m$ is the space of  $m-1$ differentiable functions with $1$-Lipschitz partial derivatives or order $m-1$ in the case where the  distribution has a bounded density. The above appears as an extension of the setting $m= r=1$  to $m=1$ and $r\!\in [1,+\infty)$.\\

 \noindent \textbf{Fundings.}  The first author benefited for this research of the support of the ``Chaire Risques Financiers'', Fondation du Risque.
 The second author thanks
the Henri Lebesgue Center (ANR-11-LABX-0020-01) and the ANR
project RAWABRANCH (ANR-23-CE40-0008).
\bibliographystyle{alpha}
\bibliography{bib_VTergo}

\appendix
\section{ Proof of Lemma~\ref{lem:tech1}.}\label{proof:lemma:2point1}
$(a)$ Let $x$, $y \!\in [0,1]^d$, $x\preceq y$  and let $x_n\to x$ such that  $x^i<x^i_n$ and $x^i_n\downarrow x^i$ for every  $i=1,\ldots,d$ such that $x^i<1$ (if $x^i=1$ for some $i$ then $y^i=1$ so that $\mu( ]\!]x,y]\!])= \nu( ]\!]x,y]\!]) =0$). We set $\tilde x^i_n = x^i_n\wedge y^i$. It is clear that
$[\![\tilde x_n,y]\!]\uparrow\, ]\!]x,y]\!]$ for the inclusion so that $\mu\big([\![\tilde x_n,y]\!]\big)\uparrow \mu( ]\!]x,y]\!])$. Idem for $\nu$. Hence
\[
|\mu( ]\!]x,y]\!])- \nu( ]\!]x,y]\!])|= \lim_n |\mu( [\![\tilde x_n,y]\!])- \nu( [\![\tilde x_n,y]\!])|\le D^{\infty}(\mu,\nu)
\]
from which we derive the announced statement.

\noindent $(b)$  Let $x$, $y \!\in [0,1]^d$, $x\preceq y$ and let $x_n\to x$ so that $x^i_n<x^i$ and  $x^i_n\uparrow x^i$ for every, $i=1,\ldots,d$  such that $x^i\neq 0$ and $x_n^i=0$ if $x^i=0$. Then
\[
]\!]x_n ,y]\!]\, \downarrow\,   K_{x,y} \mbox{ such that } K_{x,y}\cap (0,1]^d= [\![x,y]\!] \cap (0,1]^d.
\]
Then $\mu\big (]\!]x_n ,y]\!] \big)\downarrow \mu(K_{x,y})= \mu\big(K_{x,y}\cap (0,1]^d\big)= \mu\big([\![x,y]\!]\big)$. The same for $\nu$. Consequently
\[
|\mu( [\![x,y]\!])- \nu([\![x,y]\!])|= \lim_n |\mu(]\!] x_n,y]\!])- \nu(]\!] x_n,y]\!])|\le  \sup_{x,\, y\in [0,1]^d} \big| \mu\big(]\!]x,y]\!]\big) -  \nu\big(]\!]x,y]\!]\big)\big|.
\]
Combined with Claim~$(a)$ this completes the proof.\\

\noindent  (c) With the notations of $(b)$, one   checks that the sequence  $(]\![x_n,y]\!])_n$ decreases to $[\![x,y]\!]$ so that, with  the same monotone convergence argument as in $(b)$, one obtains that 
\[
D^{\infty}(\mu,\nu) \le  \sup_{x,\, y\in [0,1]^d, x\preceq y} \big| \mu\big(]\![x,y]\!]\big) -  \nu\big(]\![x,y]\!]\big)\big|,
\]
For the reverse inequality, we adapt $(a)$ by assuming that  $x_n^i=0$ when $x^i=0$. In this case, the sequence $[\![\tilde x_n,y]\!]\uparrow\, ]\![x,y]\!]$ and the sequel is identical to $(a)$.\hfill$\Box$

\end{document}